\documentclass{article}
\usepackage{graphicx} 
\usepackage{amsmath}
\usepackage{amsfonts}
\usepackage{amssymb}
\usepackage[english]{babel}
\usepackage{amsthm}
\usepackage{tikz}
\usetikzlibrary{decorations.pathreplacing,calc}
\usepackage{thm-restate}
\usepackage[margin=1.0in]{geometry}
\usepackage{hyperref}

\newtheorem{theorem}{Theorem}

\title{\textbf{Three Generalizations of Erdős Szekeres: $k$-Modal Subsequences}}
\author{
Charles Gong$^{a}$
}

\date{
$^{a}$Department of Mathematical Sciences, Carnegie Mellon University, Pittsburgh, United States \\
}
\begin{document}

\maketitle

\begin{abstract}
Erdős and Szekeres showed that given a permutation $p$ of $[n]$, and the sequence defined by \newline $(p(1), p(2), \ldots, p(n))$, there exists either a decreasing or increasing subsequence, not necessarily contiguous, of length at least $\sqrt{n}$. Fan Chung considered subsequences that can have at most one change of direction, i.e. an increasing and then decreasing subsequence, or a decreasing and then increasing subsequence. She called these unimodal subsequences, and showed there exists a unimodal subsequence of length at least $\sqrt{3n}$, up to some constants \cite{chung}. She conjectured that a permutation of $n$ contains a $k$-modal (at most $k$ changes in direction) subsequence of length at least $\sqrt{(2k+1)n}$ up to some constants. Zijian Xu proved this conjecture in 2024 \cite{xu}, and we will provide another substantially different proof using "sophisticated labeling arguments" instead of "underlying poset structures behind k-modal subsequences." We also show that there exists an increasing first $k$-modal subsequence of length at least $\sqrt{2kn}$. 
\end{abstract}


\maketitle

\section{Definitions}
We write a sequence of length $n$ as $(a_1, a_2, \ldots, a_n)$. Such a sequence is a permutation if $a_i \in [n]$ for all $i \in [n]$, and for each $b \in [n]$, there exists at least one $i$ such that $a_i = b$. A sequence $(a_1, a_2, \ldots, a_n)$ is increasing if $a_i < a_{i+1}$ for all $i \in [n-1]$. It is decreasing if $a_i > a_{i+1}$ for all $i \in [n-1]$. \\

Formally, for $k \in \mathbb{N}$ we define a $k$-modal subsequence of $(a_1, a_2, \ldots, a_n)$ as a sequence $(a_{i_1}, a_{i_2}, \ldots, a_{i_m})$ such that $i_j \in [n]$ for all $j \in [m]$, and for all $j \in [m-1]$ we have $i_j < i_{j+1}$ (i.e. it's a subsequence). Furthermore, there exists a sequence of indices $(i_{j_1}, i_{j_2}, \ldots, i_{j_p})$ where $p \leq k$ (at most $k$ indices), and $i_{j_{\ell}} \in [m]$ for all $\ell \in [p]$, and for all $\ell \in [p-1]$ we have $i_{j_{\ell}} < i_{j_{\ell+1}}$, such that for all $\ell \in [p-1]$ we have $(a_{i_1}, a_{i_2}, \ldots, a_{i_{j_1}}), (a_{i_{j_1}}, a_{i_{j_1 + 1}}, \ldots, a_{i_{j_2}}), \ldots, (a_{i_{j_p}}, a_{i_{j_p + 1}}, \ldots, a_{i_m})$ are either increasing or decreasing (i.e., at most $k$ changes in direction). We call such a subsequence increasing first if $(a_{i_1}, a_{i_2}, \ldots, a_{i_{j_1}})$ is increasing, and decreasing if $(a_{i_1}, a_{i_2}, \ldots, a_{i_{j_1}})$ is decreasing. Note a subsequence of length $1$ is both increasing and decreasing, which means in general that any $k$-modal subsequence is both an increasing first and a decreasing first $k+1$-modal subsequence. Also, the subsequences $(a_{i_1}, a_{i_2}, \ldots, a_{i_{j_1}}), (a_{i_{j_1}}, a_{i_{j_1 + 1}}, \ldots, a_{i_{j_2}}), \ldots,$ \newline $(a_{i_{j_p}}, a_{i_{j_p + 1}}, \ldots, a_{i_m})$ form the $p+1$ "parts" of the $k$-modal subsequence. \\

\section{$\sqrt{2kn}$ Increasing First $k$-modal Subsequences}

\begin{theorem}
In any permutation of $[n]$, there exists an increasing first $k$-modal subsequence of length at least $\sqrt{2kn}$, up to some constants. 
\end{theorem}

\begin{proof}
We proceed by induction. The case where $k = 0$ is trivial. We start with our base case set at $k = 1$ for illustrative purposes. \\

For the permutation $(a_1, a_2, \ldots, a_n)$, define the function $x : [n] \rightarrow [n]$ where $x(a_i)$ is the maximum length of an increasing subsequence \emph{ending} at $a_i$, including $a_i$. So for example in $(1,5,3,2,4)$, we have $x(4) = 3$ because of the subsequence $(1,3,4)$. Define $y : [n] \rightarrow [n]$ where $y(a_i)$ is the maximum length of a decreasing subsequence \emph{starting} at $a_i$, including $a_i$. Note that $x(a_i) + y(a_i) - 1$ is the length of an increasing first $1$-modal subsequence where it "changes direction" at $a_i$. \\

We claim that the function $f : [n] \rightarrow [n] \times [n]$ defined by $f(a_i) = (x(a_i), y(a_i))$ is an injection. Without loss of generality, assume $i \leq j$. Assume $(x(a_i), y(a_i)) = (x(a_j), y(a_j))$. From $x(a_i) = x(a_j)$, we know that $a_i \geq a_j$; this is because if $a_i < a_j$, we would know that $x(a_j) \geq x(a_i) + 1$, because for any increasing subsequence ending at $a_i$, we can get a longer one by appending $a_j$. Using similar reasoning, $y(a_i) = y(a_j)$ implies that $a_i \leq a_j$. This shows that $a_i = a_j$, and thus $f$ is injective. \\

Define $N$ to be the maximum length of an increasing first $1$-modal subsequence. As mentioned above, for each $a_i$ we have $x(a_i) + y(a_i) - 1 \leq N$, or $x(a_i) + y(a_i) \leq N+1$. Furthermore, we know $x(a_i), y(a_i) \geq 1$ for any $i$. Thus, $f$ is an injective function mapping $a_i$ to the triangular region defined by $x + y \leq N+1$, $x, y \geq 1$. There's at $\frac{N(N+1)}{2}$ points in this triangle. By the injectivity of $f$, we have that $n \leq \frac{N(N+1)}{2}$, and this gives us the $\sqrt{2n} \leq N$ bound up to some constants. Instead, we use the area of the triangle $n \leq \frac{N^2}{2}$ to get the bound $\sqrt{2n} \leq N$ exactly, which is potentially wrong, but will make things look nicer later. From now on, we will approximate the number of lattice points in rectangular and triangular regions by the areas of those regions. \\

Now we go on to the induction step. Assume for any permutation of length $n$, there exists an increasing first $k$-modal subsequence of length at least $\sqrt{2kn}$. We claim for any permutation of length $n$, there exists a decreasing first $k$-modal subsequence of length at least $\sqrt{2kn}$. We prove this by flipping the permutation, i.e. converting $(a_1, a_2, \ldots a_n)$ to $(n+1-a_1, n+1-a_2, \ldots, n+1-a_n)$, which is also a permutation, finding a length $\sqrt{2kn}$ increasing subsequence $(n+1-a_1, n+1-a_2, \ldots, n+1-a_n)$, and finally flipping again to get a length $\sqrt{2kn}$ decreasing subsequence in $(a_1, a_2, \ldots a_n)$. \\

Now we show there exists an increasing first $k+1$-modal subsequence of length at least $\sqrt{2(k+1)n}$. Define $x(a_i)$ to be the maximum length of an increasing subsequence ending at $a_i$. Define $y(a_i)$ to be the maximum length of a decreasing first $k$-modal subsequence starting at $a_i$. Define $N$ to be the maximum length of an increasing first $k+1$-modal subsequence. Again, note that $x(a_i) + y(a_i) - 1 \leq N$, and that using similar arguments in the $k = 1$ case, we can show $f : [n] \rightarrow [n] \times [n]$ defined by $f(a_i) = (x(a_i), y(a_i))$ is injective. \\

So again, we are mapping $a_i$ to a triangular region. We claim that there are at most $\frac{kN^2}{2(k+1)^2}$ points $(x(a_i), y(a_i))$ where $y(a_i) < \frac{kN}{k+1}$: this is because by the induction hypothesis, among the $a_i$ satisfying $y(a_i) < \frac{kN}{k+1}$, there is a decreasing first $k$-modal susbequence of length at least $\sqrt{2k (\frac{kN^2}{2(k+1)^2})} = \frac{kN}{k+1}$ (note to apply the induction hypthesis, we implicitly used a map $g$ which maps the $a_i$ satisfying $y(a_i) < \frac{kN}{k+1}$ to a permutation which preserves pairwise comparisons, i.e. $a_i < a_j$ if and only if $g(a_i) < g(a_j)$). \\

Note there are roughly $\frac{1}{2} (N - \frac{kN}{k+1})^2 = \frac{N^2}{2(k+1)^2}$ integer lattice points in the region defined by $y \geq \frac{kN}{k+1}$, $x+y \leq N+1$. This gives the bound

$$n \leq \frac{kN^2}{2(k+1)^2} + \frac{N^2}{2(k+1)^2} = \frac{N^2}{2(k+1)}$$

and so $\sqrt{2(k+1)n} \leq N$, as required. \\
\end{proof}

We now give a family of permutations that roughly achieve the $\sqrt{2kn}$ bound. We decided to describe this family, parametrized by $t$, using python code. \\

\begin{verbatim}
def strongMake(k, t):
    kstrong = []
    count = 1
    if k % 2 == 0:
        for i in range(t * (k // 2)):
            for j in reversed(range(count, count + t)):
                kstrong.append(j)
            count = count + t
    else:
        for i in range(t * (k // 2)):
            for j in reversed(range(count, count + t)):
                kstrong.append(j)
            count = count + t
        for i in reversed(range(t)):
            for j in reversed(range(count, count + i)):
                kstrong.append(j)
            count = count + i
    return kstrong
\end{verbatim}

kstrong is the permutation parametrized by $t$. The name is a remnant of Fan Chung's terminology, where she called increasing first $k$-modal subsequences as "strongly $k$-modal subsequences." As an example, consider when $k = 2$ and $t = 5$, where kstrong = $[5, 4, 3, 2, 1, 10, 9, 8, 7, 6, 15, 14, 13, 12, 11, 20, 19, 18, 17, 16, 25, 24, 23, 22, 21]$. When $(i, kstrong[i])$ is mapped onto the $x$-$y$ plane, we get \\

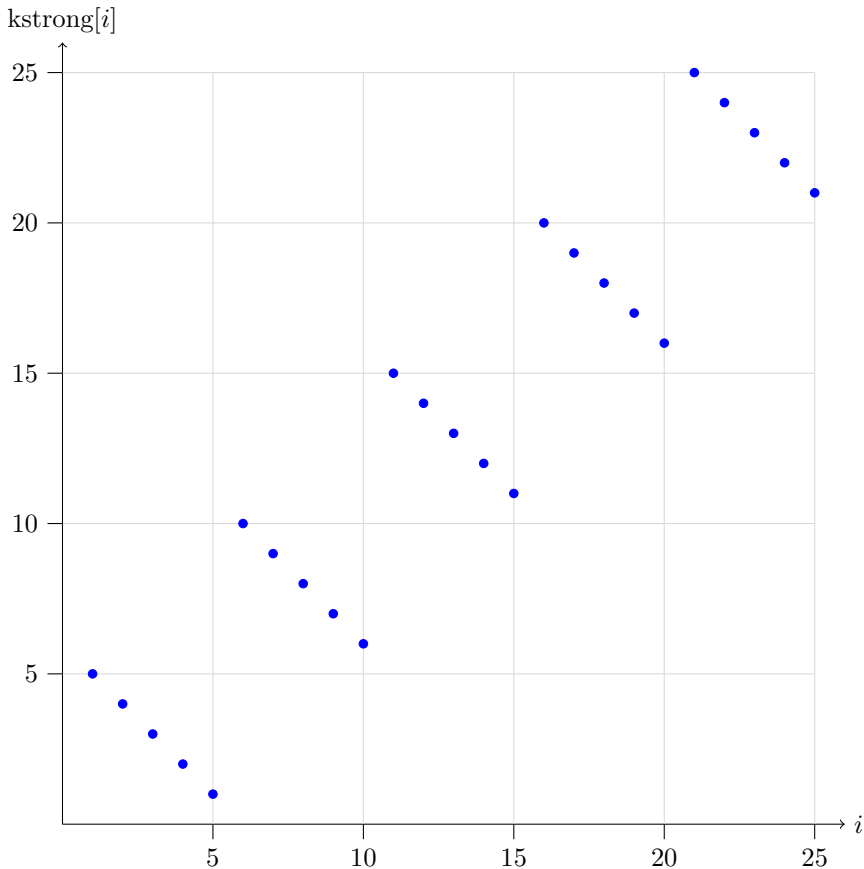
\begin{tikzpicture}[scale=0.4]

  \draw[->] (0,0) -- (26,0) node[right] {$i$};
  \draw[->] (0,0) -- (0,26) node[above] {kstrong$[i]$};

  \foreach \x in {5,10,15,20,25}
    \draw[gray!30] (\x,0) -- (\x,25);
  \foreach \y in {5,10,15,20,25}
    \draw[gray!30] (0,\y) -- (25,\y);

  \foreach \x/\y in {
    1/5, 2/4, 3/3, 4/2, 5/1,
    6/10, 7/9, 8/8, 9/7, 10/6,
    11/15, 12/14, 13/13, 14/12, 15/11,
    16/20, 17/19, 18/18, 19/17, 20/16,
    21/25, 22/24, 23/23, 24/22, 25/21}
    \filldraw[blue] (\x,\y) circle (4pt);

  \foreach \x in {5,10,15,20,25}
    \draw (\x,0) -- (\x,-0.5) node[below] {\x};
  \foreach \y in {5,10,15,20,25}
    \draw (0,\y) -- (-0.5,\y) node[left] {\y};

\end{tikzpicture}

\begin{center}
\small tikz courtesy of chatgpt
\end{center}

As you can see, kstrong consists of contiguous decreasing subsequences, such that for each pair of contiguous decreasing subsequences, all the elements in one contiguous decreasing subsequence is greater than all the elements in the other. All our constructions will share this characteristic. Recall this example above is for $k = 2$ when $t = 5$. In general, for $k = 2$ our constructions will be $t$ contiguous decreasing subsequences of length $t$. We note that in this case any increasing first $2$-modal subsequence has length at most $2t-1$: its two increasing parts can cover at most $t$ points, and its single decreasing part can also convert at most $t$ points. And so $N = 2t - 1$, $n = t^2$ and so $N = \sqrt{4n} - 1$. In the case $k$ is even, we'd have $\frac{kt}{2}$ contiguous decreasing subsequences of length $t$, and using an argument similar to the $k = 2$ case, we note that an increasing first subsequence can have length at most $\frac{kt}{2} + \frac{kt}{2} = kt$. The $\frac{k}{2}$ increasing parts can together get at most $\frac{kt}{2}$ points. The $\frac{k}{2}$ decreasing parts can each get at most $t$ points, for $\frac{k}{2}t$ points in total. And thus $N \leq kt$ with $n = \frac{kt^2}{2}$, and so $N \leq \sqrt{2nk}$. \\

In the case that $k$ is odd, we append a "triangle" consisting of $\frac{t(t-1)}{2}$ points, i.e. contiguous decreasing subsequences of length $t-1, t-2, \ldots, 1$ in that order, at the end of $\frac{(k-1)t}{2}$ contiguous decreasing subsequences of length $t$. In this case where $k$ is odd, we note that the last part is decreasing, and the second to last part is increasing. These last two parts can cover at most $t$ points of the "triangle." For the  $\frac{(k-1)t}{2}$ contiguous decreasing subsequences of length $t$, we repeat our analysis of the $k$ even case to get that $N \leq \frac{(k-1)t}{2} + \frac{(k-1)}{2}t + t = kt $. We have $n = \frac{(k-1)t^2}{2} + \frac{t(t-1)}{2} \approx \frac{kt^2}{2}$, and so $N \leq kt \approx \sqrt{2kn}$. \\

\section{$\sqrt{2kn}$ Increasing and Decreasing Subsequence Ending at the Same Point}
\begin{theorem}
In any length $n$ permutation $(a_1, a_2, \ldots, a_n)$, there exists an $a_i$ such that there exists both an increasing and a decreasing first $k$-modal subsequence of lengths at least $\sqrt{2kn}$ (up to some constants) ending at $a_i$. 
\end{theorem}

\begin{proof}
We again proceed by induction. $k = 0$ is trivial. We start at $k = 1$ for illustrative purposes. \\

Define $x(a_i)$ to be the maximum length of an increasing subsequence ending at $a_i$. Define $y(a_i)$ to be the maximum length of a decreasing subseqeuence ending at $a_i$. Note this definition of $y$ differs than the one given in Theorem 1. However, using a similar argument in Theorem 1 it can be shown that the function $f : [n] \rightarrow [n] \times [n]$ defined by $f(a_i) = (x(a_i), y(a_i))$ is injective. Define $N$ to be the maximum number such that there exists a point $a_i$ with an increasing and a decreasing $1$-modal subsequence both of length at least $N$ ending at $a_i$. Note that $0$-modal subsequences, i.e. increasing or decreasing subsequences, also count as $1$-modal subsequences. In fact, any $0$-modal subsequence is both an increasing and a decreasing first $1$-modal subsequence. So in particular $x(a_i), y(a_i) \leq N$ as well (an increasing subsequence of length $x(a_i)$ is both an increasing and a decreasing first $1$-modal subsequence of length $x(a_i)$ ending at $a_i$, so $x(a_i) \leq N$ and similarly for $y(a_i)$), and the picture to have in mind is that $f$ is mapping the points $a_i$ to a square box $\{(x,y) \in \mathbb{N}^2 : 1 \leq x, y \leq N \}$. \\

We note a key fact. Note that if there exists $(a,b) \in [N]^2$ such that 

$$|\{a_i : x(a_i) = a \text{ and } y(a_i) \leq b\}| > N+1-a$$

and 

$$|\{a_i : y(a_i) = b \text{ and } x(a_i) \leq a\}| > N+1-b$$

then there exists a point $a_i$ such that there exists both an increasing and a decreasing first $1$-modal subsequence of lengths at least $N+1$ ending at $a_i$, which would contradict the definition of $N$. We call the inequalities above \emph{the condition}. So if the condition is satisfied (I will use "triggered" interchangeably with "satisfied") for some $(a,b) \in [N]^2$, there exists a point $a_i$ such that there exists both an increasing and a decreasing first $1$-modal subsequence of lengths at least $N+1$ ending at $a_i$. \\

To see this, define $A_{(a,b)} = \{a_i : x(a_i) = a \text{ and } y(a_i) \leq b\}$ and $B_{(a,b)} = \{a_i : y(a_i) = b \text{ and } x(a_i) \leq a\}$. Now order the elements $a_i \in A_{(a,b)}$ according to $y(a_i)$, so we have $y(a_{i_1}) < y(a_{i_2}) < \ldots < y(a_{i_m})$ where $m > N+1-a$ (we know the $y$-values are distinct because they all share the same $x$-value, and $f$ is injective). Note that $(a_{i_1}, a_{i_2}, \ldots, a_{i_m})$ forms a decreasing subsequence because they share the same $x$-values. Also note by assumption $x(a_{i_1}) = a$. Thus, we know that there's an increasing first $1$-modal subsequence of length $a + m - 1 \geq a + (N+1-a) = N+1$ ending at $a_{i_m}$. Using a similar argument and ordering the elements in $B_{(a,b)}$ as $(a_{j_1}, a_{j_2}, \ldots, a_{j_p})$, there's a decreasing first $1$-modal subsequence of length $\geq N+1$ ending at $a_{j_p}$. \\

Case 1: if $i_m \leq j_p$ then we must have $a_{i_m} \geq a_{j_p}$. Otherwise assume for the sake of contradiction that $a_{i_m} < a_{j_p}$. Note that this implies $a = x(a_{i_m}) < x(a_{j_p})$ (the first equality is from $a_{i_m} \in A_{(a,b)}$). However, by assumption $x(a_{j_p}) \leq a$ because $a_{j_p} \in B_{(a,b)}$, a contradiction. Thus $a_{i_m} \geq a_{j_p}$. Note that as mentioned above, there's a decreasing first $1$-modal subsequence of length $\geq N+1$ ending at $a_{j_p}$, and also there's an increasing first $1$-modal subsequence of length $\geq N+1$ ending at $a_{i_m}$. Since $i_m \leq j_p$ and $a_{i_m} \geq a_{j_p}$, we can append $a_{j_p}$ to this increasing first $1$ modal subsequence, to get an increasing first $1$ modal subsequence of length $\geq N+1$ ending at $a_{j_p}$. \\

Case 2: if $i_m > j_p$ or $j_p < i_m$, then we must have $a_{j_p} \leq a_{i_m}$. Otherwise assume for the sake of contradiction that $a_{j_p} > a_{i_m}$ Note that this implies that $b = y(a_{j_p}) < y(a_{i_m}) \leq b$, a contradiction. Note that as mentioned above, there's an increasing first $1$-modal subsequence of length $\geq N+1$ ending at $a_{i_m}$, and also there's a decreasing first $1$-modal subsequence of length $\geq N+1$ ending at $a_{j_p}$. Since $j_p < i_m$ and $a_{j_p} \leq a_{i_m}$, we can append $a_{i_m}$ to this decreasing first $1$ modal subsequence, to get a decreasing $1$ modal subsequence of length $\geq N+1$ ending at $a_{i_m}$. \\

To recapitulate, we showed that for any $(a,b) \in [N]^2$, if $|A_{(a,b)}| > N+1-a$ and $|B_{(a,b)}| > N+1-b$, then we have an element $a_i$ which has both an increasing and decreasing first $1$ modal subsequence of length $\geq N + 1$ ending at it. \\

Now we consider the question, how many points can we place in the $N$ by $N$ square such that we avoid triggering the condition above? I.e., what is the maximum number of points we can place in the $N$ by $N$ square such that for all $(a,b) \in [N]^2$, we have either $|A_{(a,b)}| \leq N+1-a$ or $|B_{(a,b)}| \leq N+1-b$? We claim that it is $\frac{N(N+1)}{2}$, but again we approximate this by the area of the triangle defined by the convex hull of $T = \{(x,y) \in [N]^2 : x + y \leq N+1\}$ which is $\frac{N^2}{2}$. \\

We note that the set of points in the triangle $T$ does not trigger the condition, and $|T| = \frac{N(N+1)}{2}$. Now we show any set of points $S \subseteq [N]^2$ that does not trigger the condition has at most $\frac{N(N+1)}{2}$ points. \\

To see this, we claim you can fit all the points into triangle $T$ by shifting them leftward or downward. Take note of the points $(a,b) \in S$ that satisfy $|A_{(a,b)}| > N+1-a$. Note for such a point, in order to avoid triggering the condition at $(a,b)$, we must have $|B_{(a,b)}| \leq N+1-b$. Call the set of such points $C$. Consider the procedure below, which may not be possible: \\

Step 1: For each point $(x,y) \in S$ such that there exists $(a,b) \in C$ such that $x \leq a$ and $y \geq b$, shift $(x,y)$ to the left to fit it inside $T$. \\

Step 2: For all the other points, shift them down into $T$. \\

We claim that this procedure is always possible, and it shifts all the elements of $S$ into $T$. For the points that are shifted leftward in step 1, by assumption there exists $(a,b) \in C$ such that $x \leq a$ and $y \geq b$. By definition of $C$, we know that $|A_{(a,b)}| > N+1-a$. In particular, consider the point $(a, y)$. Since $|A_{(a,y)}| \geq |A_{(a,b)}| > N+1-a$ we must have $|B_{(x,y)}| \leq |B_{(a,y)}| \leq N+1-y$. In particular, this means we can shift everything to the left of $(x,y)$ also including $(x,y)$ itself into $T$. \\

Now consider $(x,y) \in S$ such that there is no $(a,b) \in C$ where $x \leq a$ and $y \geq b$. In particular, this means $(x,y) \not\in C$. Thus $|A_{(x,y)}| \leq N+1-x$, and we can shift all the points below $(x,y)$ including $(x,y)$ itself downward into $T$, but this argument only holds if we didn't shift some points leftward into column $x$ in step 1. But this can only happen if we shift some point $(x',y')$ leftward into column $x$, and there exists $(a,b) \in C$ such that $x' \leq a$ and $y' \geq b$. Note $x \leq x' \leq a$ and $y \geq y' \geq b$, contradicting our original assumption. \\

Thus, approximating the number integer lattice points in $T$ by $\frac{N^2}{2}$, we get that if the image of $f$ on $\{i \in [n] : a_i\}$ does not trigger the condition for any $(a,b)$, we have $n \leq \frac{N^2}{2}$, or $\sqrt{2n} \leq N$. Any number of points greater than $\frac{N^2}{2}$ would trigger the condition and result in both an increasing and a decreasing $1$-modal subsequence of length $\geq N+1$ ending at the same point, contradicting the definition of $N$, the maximum number such that there exists a point $a_i$ with an increasing and a decreasing $1$-modal subsequence both of length at least $N$ ending at $a_i$.  \\

Now we step into the induction step. Assume that in a permutation $(a_1, a_2, \ldots, a_n)$ there exists an $a_i$ such that there exists increasing and decreasing first $k$-modal subsequeunce of lengths at least $\sqrt{2kn}$ ending at $a_i$. We want to show there exists an $a_i$ such that there exists increasing and decreasing first $k+1$-modal subsequeunce of lengths at least $\sqrt{2(k+1)n}$ ending at $a_i$. \\

Define $x(a_i)$ to be the maximum length of an increasing $k$-modal subsequence ending at $a_i$, and $y(a_i)$ to be the maximum length of a decreasing $k$-modal subsequence ending at $a_i$. The function $f : [n] \rightarrow [n] \times [n]$ defined by $f(a_i) = (x(a_i), y(a_i))$ is an injection, which can be proved by a similar argument in theorem 1. Define $N$ as the maximum number such that there exists a point $a_i$ with an increasing and a decreasing $k+1$-modal subsequence both of length at least $N$ ending at $a_i$. Again, note a $k$-modal subsequence is both an increasing and a decreasing first $k+1$-modal subsequence, so $f$ maps the permutation into an $N$ by $N$ square. \\

By the induction hypothesis, there's at most $\frac{kN^2}{2(k+1)^2}$ points $(x(a_i), y(a_i))$ such that $x(a_i) \leq \frac{kN}{k+1}$ or $y(a_i) \leq \frac{kN}{k+1}$. This is because $\sqrt{2k (\frac{kN^2}{2(k+1)^2})} = \frac{kN}{k+1}$. Note that the region $\{(x,y) \in \mathbb{R}^2 : x \leq \frac{kN}{k+1} \text{ or } y \leq \frac{kN}{k+1}\}$ is the complement of a square $R$ in the upper right corner, $\{(x,y) \in \mathbb{R}^2 : x > \frac{kN}{k+1} \text{ and } y > \frac{kN}{k+1}\}$. That square has area $(N - \frac{kN}{k+1})^2 = \frac{N^2}{(k+1)^2}$. Using a similar argument in the $k=1$ case, one can show that if the condition is satisfied for some $(a,b) \in [N]^2$, there exists a point $a_i$ such that there exists both an increasing and a decreasing first $k+1$-modal subsequence of lengths at least $N+1$ ending at $a_i$, contradicting the definition of $N$, the maximum number such that there exists a point $a_i$ with an increasing and a decreasing $k+1$-modal subsequence both of length at least $N$ ending at $a_i$. Note that any subset $S \subseteq R$ satisfies a similar conditional implication, and thus using a similar argument we get that if the condition is not triggered, there is at most $\frac{1}{2}(N - \frac{kN}{k+1})^2 = \frac{N^2}{2(k+1)^2}$ points $(x(a_i), y(a_i))$ in $R$. Thus the number of points in total $n$ satisfies the inequality

$$n \leq \frac{kN^2}{2(k+1)^2} + \frac{1}{2}\frac{N^2}{(k+1)^2} = \frac{N^2}{2(k+1)}$$

and so $\sqrt{2(k+1)n} \leq N$. 
\end{proof} 

Note that Theorem 2 implies Theorem 1, but we included a direct proof of Theorem 1 because we thought it was cool. Also, the family of constructions that attains this lower bound is the same as in Theorem 1. \\

\section{$\sqrt{(2k+1)n}$ $k$-modal Subsequence}

\begin{theorem}
In any permutation of $[n]$, there exists a $k$-modal subsequence (either increasing or decreasing first) of length at least $\sqrt{(2k+1)n}$. 
\end{theorem}

\begin{proof}
We do NOT proceed by induction. Fix $k \in \mathbb{N}$. Define $x(a_i)$ to be the maximum length of an increasing first $k$-modal subsequence ending at $a_i$, and $y(a_i)$ to be the maximum length of a decreasing first $k$-modal subsequence ending at $a_i$. Define $f : [n] \rightarrow [n] \times [n]$ to be $f(a_i) = f(x(a_i), y(a_i))$. One can show $f$ is injective by a similar argument in Theorem 1. Define $N$ to be the maximum length of a $k$-modal subsequence, either increasing or decreasing first. Note $f$ maps $a_i$ into an $N$ by $N$ square. \\

By theorem 2, there exists at most $\frac{2kN^2}{(2k+1)^2}$ points $(x(a_i), y(a_i))$ such that $x(a_i) < \frac{2kN}{2k+1}$ or $y(a_i) < \frac{2kN}{2k+1}$. This is because $\sqrt{2k(\frac{2kN^2}{(2k+1)^2})} = \frac{2kN}{2k+1}$, and thus if we had more points that there's exist $a_i$ such that $x(a_i) \geq \frac{2kN}{2k+1}$ and $y(a_i) \geq \frac{2kN}{2k+1}$. And note, there's at most $(N-\frac{2kN}{2k+1})^2 = \frac{N^2}{(2k+1)^2}$. Thus, we have 

$$n \leq \frac{2kN^2}{(2k+1)^2} + \frac{N^2}{(2k+1)^2} = \frac{N^2}{2k+1}$$

and so $\sqrt{(2k+1)n} \leq N$, as required.
\end{proof}

We now construct a family of constructions that attains this lower bound. Consider the python code below. \\

\begin{verbatim}
def permMake(k, t):
    kperm = []
    count = 1
    for i in range(t):
        for j in reversed(range(count, count + t + i)):
            kperm.append(j)
        count = count + t + i
    for i in range((k-1)*t):
        for j in reversed(range(count, count + 2*t)):
            kperm.append(j)
        count = count + 2*t
    for i in reversed(range(t)):
        for j in reversed(range(count, count + t + i)):
            kperm.append(j)
        count = count + t + i
    return kperm
\end{verbatim}

To describe this in words, we have contiguous decreasing subsequences of lengths 

$$t, t+1, \ldots, 2t-1, 2t, 2t, \ldots, 2t, 2t-1, 2t-2, \ldots t,$$

where we have $(k-1)t$ contiguous decreasing subsequences of length $2t$ in total. This gives us $n \approx \frac{3t^2}{2} + (k-1)2t^2 + \frac{3t^2}{2} = (2k+1)t^2$. By casing on $k$ even and odd, one can show in either case that the maximum length of a $k$-modal subsequence is $N \leq (2k+1)t$, and so $N \leq \sqrt{(2k+1)n}$.

\section{Comments}
After writing most of this paper on August 24, 2025, we discovered that Zijian Xu has already proved Fan Chung's conjecture before us in the March of 2024 \cite{xu}. It seems like our proofs are substantially different, as our method involves "sophisticated labeling arguments" and theirs use "underlying poset structures behind k-modal subsequences." \cite{xu} We do believe, however, we were first to discover a proof that increasing subsequences have length at least $\sqrt{2kn}$. \\

On page 20 of Zijian Xu's proof \cite{xu}, we believe the proof still goes through but they forgot to consider the case if $d'(R') \neq d(R)$. On page 8 they possibly meant that the $(i-k)$-th section is given by $P_0^{i}$, judging by observation 10. \\

We believe Fan Chung's proof of the $k = 1$ case in \cite{chung} can be simplified. She did not have to define $f$ and $g$, we believe. $\lambda$ can be extended to cover all points while remaining injective, and the points fit in the region $N_1 + N_2 - N \leq x+y \leq N$ and $1 \leq x \leq N_1, 1 \leq y \leq N_2$. 

\section{Acknowledgements}
Rika Furude provided great advice for this research nipah $\sim$! We would also like to thank Edward Hou for their feedback. This problem was introduced to me by Boris Bukh. \\

\bibliographystyle{plain}

\end{document}